\documentclass[11pt]{amsart}
\usepackage{amsmath,amsfonts,verbatim,hyperref}

\begin{document}
\bibliographystyle{plain}

\newtheorem{thm}{Theorem}[section]
\newtheorem{lem}[thm]{Lemma}
\newtheorem{prop}[thm]{Proposition}
\newtheorem{cor}[thm]{Corollary}
\newtheorem{conj}[thm]{Conjecture}
\newtheorem{defn}[thm]{Definition}

\newcommand{\tr}{\ensuremath{\operatorname{tr} } }
\newcommand{\AH}{\ensuremath{\operatorname{AH} } }
\newcommand{\AHbar}{\ensuremath{\operatorname{A\overline{H} } } }
\newcommand{\AI}{\ensuremath{\operatorname{A\mathcal{I} } } }
\newcommand{\AIbar}{\ensuremath{\operatorname{A\overline{\mathcal{I} } } } }

\renewcommand{\H}{\ensuremath{\operatorname{H} } }
\newcommand{\Hbar}{\ensuremath{\operatorname{\overline{H} } } }
\newcommand{\GH}{\ensuremath{\operatorname{GH} } }
\newcommand{\GHbar}{\ensuremath{\operatorname{G\overline{H} } } }
\newcommand{\QH}{\ensuremath{\operatorname{QH} } }
\newcommand{\QHbar}{\ensuremath{\operatorname{Q\overline{H} } } }
\newcommand{\GI}{\ensuremath{\operatorname{G\mathcal{I} } } }
\newcommand{\GIbar}{\ensuremath{\operatorname{G\overline{\mathcal{I} } } } }
\newcommand{\QI}{\ensuremath{\operatorname{Q\mathcal{I} } } }
\newcommand{\QIbar}{\ensuremath{\operatorname{Q\overline{\mathcal{I} } } } }
\newcommand{\I}{\ensuremath{\operatorname{\mathcal{I} } } }
\newcommand{\Ibar}{\ensuremath{\operatorname{\overline{\mathcal{I} } } } }

\newcommand{\PSLC}{\ensuremath{{\rm PSL}_2 ({\bf C})} }
\newcommand{\Hom}{\ensuremath{\operatorname{Hom} } }
\newcommand{\Mod}{\ensuremath{\operatorname{Mod} } }
\newcommand{\QF}{\ensuremath{\operatorname{QF} } }

\title{The curious moduli spaces of unmarked {Kleinian} surface groups}
\author{Richard D. Canary and Peter A. Storm}

\begin{abstract}
Fixing a closed hyperbolic surface $S$, we define a moduli space $\AI(S)$ of unmarked hyperbolic $3$-manifolds homotopy equivalent to $S$.  This $3$-dimensional analogue of the moduli space $\mathcal{M}(S)$ of unmarked hyperbolic surfaces homeomorphic to $S$ has bizarre local topology, possessing many points that are not closed.  There is, however, a natural embedding $\iota: \mathcal{M}(S) \to \AI(S)$ and compactification $\AIbar(S)$ such that $\iota$ extends to an embedding of the Deligne-Mumford compactification $\overline{\mathcal{M}}(S) \to \AIbar(S)$.
\end{abstract}

\thanks{Canary was partially supported by NSF grants DMS-0504791 and DMS-0554239. Storm was partially supported by NSF grant DMS-0741604
and the Roberta and Stanley Bogen Visiting Professorship at Hebrew University.}
\date{\today}

\maketitle

\section*{}\label{sec:introduction}
\setcounter{section}{1} 

For a closed oriented hyperbolic surface $S$, the moduli space $\mathcal{M}(S)$ of hyperbolic (or Riemann) surfaces homeomorphic to $S$ is a familiar and interesting object.  Hyperbolic $3$-manifolds homotopy equivalent to $S$ have also been intensely studied since the foundational work of Ahlfors and Bers.  An analogous moduli space of (unmarked) hyperbolic $3$-manifolds homotopy equivalent to $S$ is rarely considered.  Here we define such a topological moduli space $\AI(S)$, study a bit of its local topology, and define a natural  compactification $\AIbar(S)$. 

Perhaps a reason for its anonymity, the local topology of $\AI(S)$ is bizarre.  We can reinterpret interesting phenomena from the theory of Kleinian groups as statements about $\AI(S)$.  There exist singly-degenerate Kleinian groups whose ending lamination is fixed by a pseudo-Anosov homeomorphism.  This beautiful fact implies $\AI(S)$ is not well separated, possessing many points that are not closed.  Compactness of Bers slices implies that even the geometrically finite points fail to form a Hausdorff space. 

As penance for its local foibles, $\AI(S)$ offers better global topological behavior.  Recall the Deligne-Mumford compactification $\overline{\mathcal{M}}(S)$ is obtained by adding Riemann surfaces with nodes.  The space $\AI(S)$ has an analogous compactification $\AIbar(S)$ obtained by adding hyperbolic $3$-manifolds $N$ homotopy equivalent to $S$ cut along simple closed curves, such that the corresponding ``cut'' homotopy classes in $N$ are parabolic. 
We begin by constructing an augmented deformation space $\AHbar(S)$, which is the analogue of the augmented
Teichm\"uller space, and then show that its quotient $\AIbar(S)$
is sequentially compact.
Sequential compactness follows by combining theorems of Thurston \cite[Thm.6.2]{Th4} and Canary-Minsky-Taylor \cite{CMT} to show that, up to re-marking, a sequence in $\AH (S)$ always converges on the complement of a multicurve in $S$.

There is an embedding $\mathcal{M}(S) \to \AI(S)$ with image the closed set of Fuchsian manifolds.  This embedding extends to an embedding $\overline{\mathcal{M}}(S) \to \AIbar(S)$ of the Deligne-Mumford compactification.
Hopefully connoisseurs of the more familiar $\mathcal{M}(S)$ will find it interesting to see how the story changes abruptly in $3$-dimensions, with local structure suffering while global compactness manages to survive.  

\medskip\noindent
{\bf Acknowledgements:} The authors would like to thank the referee for a very
careful reading of the original manuscript and many helpful suggestions to improve the
exposition.

\section{Preliminaries}\label{sec:prelim}

Let $S$ be a compact oriented hyperbolizable surface with (possibly empty) boundary.  We do not assume $S$ is connected.  Denote the set of components of $S$ by $c(S)$.  For convenience, put a hyperbolic metric on $S$ so that any boundary curves are very short horocycles.  In this case $S$ embeds isometrically into a complete finite area hyperbolic surface $\hat{S}$ homeomorphic to the interior of $S$.  By choosing the boundary curves to be sufficiently short, we can assume the simple closed geodesics of $\hat{S}$ are contained in $S$.  A multicurve $b$ is a (necessarily finite) set of pairwise disjoint simple closed geodesics on $S$.  (As geodesics, no component of $b$ is homotopic into $\partial S$.)  Let $\mathcal{N} b \subset S$ denote a small open collar neighborhood of $b$.

All hyperbolic $3$-manifolds in this article are assumed to be oriented and complete.  Unless otherwise stated, they have no boundary.  However, our manifolds will not always be connected.  Let $\mu_3>0$ be the Margulis constant for hyperbolic $3$-manifolds.  Recall that $\mu_3$ is a number with the following property.  If $N$ is any hyperbolic $3$-manifold then the subset of points in $N$ with injectivity radius at most $\mu_3$ is a disjoint union of two types of sets: an embedded solid torus neighborhood of a simple closed geodesic with length less than $\mu_3$, or a properly embedded product region $T \times [0,\infty)$ where $T$ is either a torus or a noncompact annulus without boundary \cite{BP}.  Define the cusps of $N$ to be the components of the form $T \times [0,\infty)$.

We now give two definitions of the set $\H(S,\partial S)$ of marked Kleinian surface groups associated to $S$; one algebraic and one geometric.  These two perspectives are both useful.  Many of the following definitions will have an algebraic and a geometric formulation.  We begin with the algebraic definition.
Let $P \subseteq \partial S$ be a subset of boundary components.  (We are interested only in the two cases $P=\partial S$ and $P=\emptyset$.)  When $S$ is connected, the set $\H(S,P)$ of marked Kleinian surface groups associated to $S$ is the set of conjugacy classes of discrete faithful representations 
$\rho: \pi_1(S)\to \PSLC$ such that if
$\gamma \in \pi_1(S)$ can be represented by a curve freely homotopic into $P$, then $\rho(\gamma)$ is parabolic.
If $S$ is disconnected, $\H(S, P)$ is the Cartesian product of
$\H(R, P \cap R)$ over the components $R$ of $S$.  The notation $\H(S)$ is shorthand for $\H(S,\emptyset)$.

Abusing notation slightly, we will often simply use $\rho$ to denote an element of $\H(S,P)$.  If necessary, we will distinguish the representations of the components of $S$ using the notation $\{ \rho_R \}_{R \in c(S)}$, where $c(S)$ is the set of components of $S$.  An element $\rho$ of $\H(S,P)$ determines a hyperbolic $3$-manifold 
\[ N_\rho = \coprod_{R \in c(S)} {\bf H}^3 / \rho_R (\pi_1 (R)),\]
together with a homotopy equivalence $m_\rho: S \to N_\rho$, known as the marking of $N_\rho$, induced by the isomorphism $\pi_1 (S) \to N_\rho$ (with the obvious modification for disconnected $S$).  The marking sends $P$ into the cusps of $N_\rho$.  The phrase, ``a hyperbolic manifold in $\H(S,\partial S)$," indicates such a pair $(N_\rho, m_\rho)$.  

With this pair of objects in mind, we may define $\H(S, P)$ more geometrically as follows.  For an oriented hyperbolic $3$-manifold $N$ and a homotopy equivalence $m: S \to N$, $(N,m) \in \H(S,P)$ if $m$ takes $P$ into the cusps of $N$.  Two pairs $(N_1, m_1)$ and $(N_2, m_2)$ are equal in $\H(S ,P)$ if there exists an orientation preserving isometry $\iota: N_1 \to N_2$ such that $\iota \circ m_1 $ is homotopic to $m_2$.  

By Bonahon's tameness theorem, any hyperbolic manifold in $\H(S,\partial S)$ is homeomorphic to $S \times {\bf R}$ \cite{BoTame}.  We could ``stiffen'' the above  geometric definitions using homeomorphisms rather than homotopy equivalences.  Nothing is gained from this, and we will keep the above more traditional definitions.

An element $\{\rho_R \}_{R \in c(S)}$ of $\H(S, \partial S)$ is Fuchsian if each $\rho_R$ is conjugate to a representation with image in the group ${\rm PSL}_2 ({\bf R})$ of isometries of the hyperbolic plane.  Equivalently, a hyperbolic manifold in $\H(S ,\partial S)$ is Fuchsian if it contains an embedded totally geodesic hyperbolic surface homotopy equivalent to $N$.

Define a topology on $\H(S,P)$ as follows.  Assume first that $S$ is connected.  Consider the subset $\Hom_P (\pi_1 (S), \PSLC) \subseteq \Hom (\pi_1 (S), \PSLC)$ of homomorphisms taking conjugacy classes of $P$ to parabolic elements of $\PSLC$.  Put the compact-open topology on $\Hom_P (\pi_1 (S), \PSLC)$.  Let $\mathcal{D}$ denote the subset of faithful homomorphisms with discrete image.  (The set $\mathcal{D}$ is closed \cite{chuckrow,jorgensen}.)  The isometry group $\PSLC$ acts by conjugation on $\Hom_P (\pi_1 (S), \PSLC)$.  There is a subset $\mathcal{O}$ of $\Hom_P (\pi_1 (S), \PSLC)$ containing $\mathcal{D}$ such that the quotient of $\mathcal{O}$ by this action is a smooth complex manifold.  (The set $\mathcal{O}$ is the set of nonradical representations.  See \cite[Sec. 4.3]{K}.)  The subset $\mathcal{D}$ is preserved by the $\PSLC$-action, and it makes sense to define the topological quotient
\[ \AH (S,P) := \frac{\mathcal{D} \cap \Hom_P (\pi_1 (S), \PSLC)}{\PSLC}. \]
$\AH(S,P)$ is a subspace of the manifold $\mathcal{O}/\PSLC$, which in turn can be algebraically ``completed'' to the character variety $X(S,P)$, an algebraic variety which is intuitively the quotient of $\Hom_P (\pi_1(S), \PSLC)$ by the $\PSLC$-action.
For more information on this construction see \cite{HP,K}.
If $S$ is not connected then topologize $\H(S,P)$ as the topological product of the $\H(R, P\cap R)$ for $R \in c(S)$.  This topology on $\AH (S,P)$ is called the algebraic topology.

More geometrically, we say that a sequence $\{(N_n,m_n)\}$ in $\text{H}(S,\partial S)$
converges algebraically to $(N,m)$ if there exist homotopy equivalences $h_n:N\to N_n$
such that $m_n$ is homotopic to $h_n\circ m$ for all $n$ and $\{h_n\}$ $C^\infty$-converges  to
a local isometry on every compact subset of $N$.  (The equivalence of these definitions
is discussed in \cite[Sec. 3.1]{Mc3}.)

The interior $\QF(S)$ of $\AH (S,\partial S)$ (as a subset of $X(S,\partial S)$) consists
of the quasifuchsian hyperbolic 3-manifolds \cite{Marden,SullivanII}.
We recall that if $N={\bf H}^3/\Gamma$
then its conformal boundary $\partial_c N$ is the quotient of the domain of
discontinuity for the action of $\Gamma$ on the Riemann sphere. 
A hyperbolic 3-manifold $N$ in $\AH(S,\partial S)$ is quasifuchsian if
the conformal bordification $N\cup\partial_c N$ of $N$ is homeomorphic to $S\times [0,1]$.
Bers \cite{Ber7} showed that a quasifuchsian hyperbolic 3-manifold is determined by the
conformal structure on $S\times\{0,1\}$ and that any conformal structure arises.
If $(X,Y)\in \mathcal{T}(S)\times \mathcal{T}(\overline{S})$, then we
let $Q(X,Y)\in \AH(S,\partial S)$ be the quasifuchsian hyperbolic 3-manifold
with conformal structure $X$ on ``top'' and conformal structure $Y$ on the
``bottom.'' (Here $\overline{S}$ denotes $S$ with the opposite orientation.)

\section{The topology of the moduli space}

We will assume throughout this section that $S$ is not a thrice-punctured sphere,
since in that case $\H(S,\partial S)$ is a point.
The moduli set $\I(S,\partial S)$ of unmarked Kleinian surface groups
is simply the quotient of $\H(S,\partial S)$ by the natural action of the mapping
class group $\Mod(S)$ of isotopy classes of orientation-preserving
homeomorphisms of $S$. We recall that if $\phi$ is a (representative of a)
mapping class in
$\Mod(S)$, then $\phi$ acts on $\AH(S,\partial S)$ by taking
$\rho$ to $\rho\circ\phi_*^{-1}$.
An element of $\I(S,\partial S)$ is simply an oriented hyperbolic 3-manifold $N$ which is
homotopy equivalent to $S$ (by a homotopy equivalence which takes $\partial S$ into
the cusps of $N$), where we do not keep track of the specific homotopy equivalence.

The moduli space $\I(S,\partial S)$ inherits an algebraic topology
which we denote by $\AI(S,\partial S)$. 
In the algebraic topology, this moduli space is rather badly
behaved topologically. A first hint that this should be the case is the observation, first
due to Thurston \cite{Th4} (see also McMullen \cite{Mc3}), that there are points in
$\H(S,\partial S)$ which are fixed by infinite order elements of $\Mod(S)$. These points arise naturally
as  the covers associated to the fibers of finite volume hyperbolic 3-manifolds which fiber over the circle.
We recall the outline of Thurston's construction.
If $\phi:S\to S$ is a pseudo-Anosov element of $\Mod(S)$ and
$(X,Y)\in \mathcal{T}(S)\times \mathcal{T}(\overline{S})$, then Thurston considers
the sequence  of quasifuchsian Kleinian groups $\rho_n = Q(\phi^{n}(X),\phi^{-n}(Y))$ in $\AH (S, \partial S)$.  He shows
that $\{\rho_n\}$ converges to a Kleinian group $\rho$ with empty
domain of discontinuity which is a fixed point of the action of $\phi$ on
$\AH(S,\partial S)$.

A minor variation on Thurston's construction allows us to construct points in
$\AI(S,\partial S)$ that are not closed.

\begin{prop} \label{T_1-separated}
There are points of $\AI (S,\partial S)$ that are not closed.
\end{prop}

\begin{proof}
It suffices to consider the case when $S$ is connected.
Pick a pseudo-Anosov homeomorphism $\phi : S \to S$ and
$(X,Y)\in{\mathcal{T}}(S)\times \mathcal{T}(\overline{S})$.
If we let $\rho_n = Q(X,\phi^{-n}(Y))$, then Bers \cite[Thm.3]{Ber1} proved that $\{\rho_n\}$
has a convergent subsequence with limit $\rho \in \H(S,\partial S)$ which has non-empty
domain of discontinuity.
Consider the sequence
$\rho \circ \phi_*^n \subset \H(S,\partial S)$.  Using Thurston's Double 
Limit Theorem \cite{Th4}, one can prove that up to subsequence this converges
to a manifold with an empty domain of discontinuity \cite[3.11]{Mc3}.  
Therefore the fiber in $\AH(S,\partial S)$ over the image of
$\rho$ in $\AI(S,\partial S)$ is not closed, implying that the image of $\rho$ is not a closed point.
\end{proof}

\medskip\noindent
{\bf Remark:} One may further show that if $N$ is a  degenerate
hyperbolic 3-manifold  in $\AI(S,\partial S)$ with a lower bound on its injectivity radius,
then $N$ is not a closed point in $\AI(S,\partial S)$. (More generally, one need only
assume that there is a lower bound on the injectivity radius of $N$ outside of
cusps associated to $\partial S$.)  We recall that $N$ is degenerate if its
domain of discontinuity has exactly one component and every cusp in $N$
is associated to a component of $\partial S$.
In this case, $N$ has one geometrically infinite end and there exists a sequence 
$\{ h_n:S\to N\}$ of pleated surfaces, each of which is a homotopy equivalence, exiting the geometrically infinite end of $N$.
Since $h_n(S)$ has bounded geometry (away from the cusps),
one may re-mark the maps $h_n$ so that the associated sequence
of representations $\rho_n=(h_n)_*$ in $AH(S,\partial S)$ has a convergent
subsequence with limit $\rho$ such that the domain of discontinuity of
$\rho(\pi_1(S))$ is empty. Again, it follows that $N$ is not a closed point
of $\AI(S,\partial S)$.

\bigskip

One can show that at least the geometrically finite points in $\AI(S,\partial S)$ are  closed.

\begin{prop}\label{prop:T1}
If $N \in \I(S,\partial S)$ is geometrically finite then $N$ is a closed point in 
$\AI(S,\partial S)$.
\end{prop}

\begin{proof}
Let $G$ be a graph in $S$ which is a bouquet of circles associated to a 
minimal generating set of $\pi_1(S)$.
Given  an element $(N,m)$ in the pre-image of $N$ under the quotient map
$\H(S,\partial S) \to \I(S,\partial S)$ one obtains a
graph $m(G)$ in $N$,
and the element $(N,m)$ is determined by the homotopy
class of this marked graph. If $\{(N,m_n)\}$ is an algebraically convergent sequence of elements of
the fiber then one may assume that $m_n (G)$ has length less than $K$ for some $K$
(perhaps after homotoping the maps $m_n$). It follows that there exist constants $0< \varepsilon_0 <\varepsilon_1$ 
so that if $x\in m_n(G)$ (for any $n$), then $\operatorname{inj}_N (x)\in [\varepsilon_0,\varepsilon_1]$. 
(The existence of $\varepsilon_1$ is obvious.  If curves penetrate arbitrarily deeply into the thin part, but are not contained entirely within the thin part, then they must be growing arbitrarily long, hence the lower bound $\varepsilon_0$.)
Let $C$ be the set of points in $N$ with injectivity radius in $[\varepsilon_0,\varepsilon_1]$.
Since $N$ is geometrically finite, $C$ is compact. Therefore, there exists at most
finitely many homotopy classes of immersions of $G$ into $C$ with total length at most
$K$.  It follows that there are at most finitely many distinct elements in the sequence
$\{(N,m_n) \}$, so the sequence is eventually constant.
It follows that the pre-image of $N$ is closed, and hence that $N$ is a closed point.
\end{proof}

Surprisingly, the set of geometrically finite points in $\AI(S,\partial S)$ is not Hausdorff. 

\begin{prop}\label{prop:!T2}
The geometrically finite points of $\AI (S,\partial S)$ do not form a Hausdorff space in the subspace topology.
\end{prop}

\begin{proof} 
Let $D \in \Mod(S)$ be a Dehn twist about a non-peripheral simple closed curve on $S$.  Pick
$(X,Y)\in\mathcal{T}(S)\times\mathcal{T}(\overline{S})$ such that $X$ and $Y$ are not isometric.
Let $\rho_n = Q(X, D^n(Y))$ for all $n$.  
By  \cite{Ber1} and \cite[Sec.3]{KT}, $\{\rho_n\}$ converges algebraically to a manifold whose ``top'' conformal boundary component is isometric to $X$, and whose ``bottom'' conformal boundary has developed a rank one cusp.  Similarly, $\rho_n \circ D_*^{-n} = Q(D^{-n}(X), Y)$ converges algebraically to a manifold whose ``bottom'' conformal boundary component is isometric to $Y$.  In particular, the two limiting manifolds cannot be isometric, thus building a sequence in $\AI (S,\partial S)$ with two distinct limits.
\end{proof}

\medskip\noindent
{\bf Historical Remarks.} It follows from Bers' simultaneous uniformization \cite{Ber7} that the mapping class group $\Mod(S)$ acts properly
discontinuously, but not freely, on the interior $\QF(S)$ of $\AH(S,\partial S)$.
If we parameterize $\QF(S)$ as $\mathcal{T}(S)\times\mathcal{T}(\overline{S})$, then
the action is just the diagonal action where $\Mod(S)$ acts on each factor
in the usual manner. Its quotient is thus naturally a bundle, in the orbifold sense, over the moduli space $\mathcal{M}(S)$
with generic fiber homeomorphic to $\mathcal{T}(S)$.  
It has recently been shown \cite{BCM} that $\AH(S,\partial S)$ is the closure of $\QF(S)$
(see also \cite{Bromberg,Brock-Bromberg}).
The examples above (based on the cited work of others) show that $\Mod(S)$ does not act properly discontinuously on 
$\AH(S,\partial S)$ and hence does not act properly discontinuously
on $X(S,\partial S)$. Moreover, Souto and Storm \cite{St6} showed that $\Mod(S)$ acts 
topologically transitively on the closure of the set 
of points in the frontier of $\QF(S)$ whose conformal
boundary does not contain a component homeomorphic to $S$.
One can further show
that any open, $\Mod(S)$-invariant open subset of $X(S,\partial S)$ on which
$\Mod(S)$ acts properly discontinuously is disjoint from $\partial QF(S)$ (see
Lee \cite{lee}). It is conjectured
that every $\Mod(S)$-invariant open subset of $X(S,\partial S)$
on which $\Mod(S)$ acts properly
discontinuously is a subset of $\QF(S)$.

In the case that $S$ is a once-punctured torus, Bowditch \cite{Bow2} studied
the subset $\Phi_Q$ of $X(S,\partial S)$ consisting of representations $\rho$
such that if $\Omega$
is the set of simple closed curves on $S$, then $\rho(\gamma)$ is hyperbolic
for all $\gamma\in\Omega$ and there are only finitely many elements of
$\Omega$ such that $|\tr^2(\rho(\gamma))|\le 4$.   He conjectured that
$\Phi_Q=\QF(S)$. Tan, Wong and Zhang
\cite{TWZ} showed that $\Phi_Q$ is an open subset of $X(S,\partial S)$ on which
$\text{Mod}(S)$ acts properly discontinuously.  Cantat \cite{Cantat} has
used techniques from holomorphic dynamics to investigate the action of the mapping class group on
the character variety associated to the once-punctured torus.

\section{The augmented deformation set}\label{sec:augmented}

The goal of this section is to enlarge $\H(S, \partial S)$ into an augmented deformation set $\Hbar (S, \partial S)$.  The action of $\Mod (S)$ will extend to an action on 
$\Hbar(S, \partial S)$ and its quotient will give a compactification of
$\AI (S,\partial S)$.  The augmentation described here is analogous to the augmentation of Teichm{\"u}ller space by adding noded Riemann surfaces.

We again give two definitions of our augmented deformation set, with algebraic and geometric flavors.  We begin with the algebraic definition.
If $a$ is a multicurve on $S$, let $c(a)$ denote the collection of components
of $S-\mathcal{N}a$ where $\mathcal{N}a$ is an open collar neighborhood of
$a$ in $S$. An element 
$(\{\rho_R\}_{R\in c(a)},a)\in \Hbar(S, \partial S)$
is a multicurve $a$ together with an element
$\rho_R\in \H(R,\partial R)$ for each $R\in c(A)$.  An element of $\Hbar(S,\partial S)$ is geometrically finite (resp. Fuchsian) if each $\rho_R$ is geometrically finite (resp. Fuchsian).  Elements of $\Hbar(S,\partial S)$ using the multicurve $a$ define the stratum corresponding to $a$.

The geometric definition is somewhat more difficult to formulate, but it gives
more insight into the nature of elements of the set. 
Let $\mathcal{U}$ be the set of triples $(N, a, m)$ where:
\begin{enumerate}
\item  $N$ is a (possibly disconnected) oriented hyperbolic $3$-manifold.
\item\label{fragile}  $a \subset S$ is a multicurve with open collar 
neighborhood $\mathcal{N} a$.
\item  $m$ is a homotopy equivalence
\[m:   (S - \mathcal{N} a) \to N\]
taking $\partial(S - \mathcal{N} a)$ into the cusps of $N$.  
\end{enumerate}
A triple $(N,a,m)$ is geometrically finite if each component of $N$ is geometrically finite.  Similarly, it is Fuchsian if every component of $N$ is Fuchsian.

Of course this set $\mathcal{U}$ is too large.  Form an equivalence relation by declaring two elements $(L , a, \ell)$ and $(N, b, m)$ of $\mathcal{U}$ to be equivalent if $a = b$ and there exists an orientation preserving isometry $\iota: L \to N$ such that 
\[ m^{-1} \circ \iota \circ \ell: (S - \mathcal{N} a) \to (S - \mathcal{N} a)\]
is homotopic to the identity.  (The map $m^{-1}$ is a homotopy inverse of $m$.)
The augmented deformation set
$\Hbar (S,\partial S)$ is $\mathcal{U}$ modulo the above equivalence relation.

\section{The algebraic topology on the augmented deformation set}

We now define the algebraic topology for $\Hbar (S,\partial S)$, which extends the algebraic topology on $\H(S ,\partial S)$.  Its definition is motivated by Thurston's notion of a maximal subsurface of convergence.

\begin{thm} \label{owbt} \cite[Thm. 6.2]{Th4}
Given any sequence $\{ \rho_n \} \subset \AH(S, \partial S)$ there exists a subsequence $\{\rho_{j}\}$ and a (possibly empty, possibly disconnected) subsurface $R$ of $S$ with incompressible boundary such that:
\begin{enumerate}
\item For each component $F$ of $R$ the sequence of restrictions $\{\rho_{j}|_F\}$ converges in $\AH(F)$.
\item If $\Gamma$ is a nontrivial subgroup of $\pi_1 (S)$ such that the restrictions
$\{\rho_{j}|_\Gamma\}$ converge (up to conjugacy) on a subsequence of $\{\rho_{j}\}$ then $\Gamma$ is conjugate to a subgroup of $\pi_1 (F)$ for some component $F$ of $R$.
\end{enumerate}
\end{thm}

We will be particularly interested in the special case of this theorem where the subsequence is the entire sequence and $R$ is the complement of an open collar neighborhood of a multicurve.

Let $\mathcal{C}$ be the set of conjugacy classes of $\pi_1 (S)$. For 
an element $(\{\rho_R \}_{R \in c(a)}, a) \in \Hbar(S,\partial S)$ and $\gamma \in \mathcal{C}$ we will define the square of  the trace of an element $\gamma$ in $(\{\rho_R \}_{R \in c(a)})$ as an element of the Riemann sphere ${\bf CP}^1$. 
If for some $R' \in c(a)$, $\gamma$ can be realized by a closed curve in $R'$, then define $\tr^2((\{\rho_R \}_{R \in c(a)},a), \gamma)$ to be the square of the trace of 
$\rho_{R'} (\gamma)$. (Note that the trace of an element of $\PSLC$ is not well-defined
but the square of its trace is well-defined.) Otherwise, define $\tr^2((\{\rho_R \}_{R \in c(a)},a), \gamma)$ to be $\infty$.  This case occurs if and only if $\gamma$ essentially intersects the multicurve $a$.  Using this extended length function we define the set map
\begin{align*}
t_* : \Hbar(S,\partial S) &\to ({\bf CP}^1)^\mathcal{C} \\
(\{\rho_R \}_{R \in c(a)},a) & \mapsto \left\{ \gamma \mapsto \tr^2((\{\rho_R \}_{R \in c(a)},a), \gamma) \right\},
\end{align*}
where $({\bf CP}^1)^\mathcal{C}$ is the space of functions from $\mathcal{C}$ to 
${\bf CP}^1$ in the product topology (which is equivalent to pointwise convergence).

\begin{lem}\label{lem:inj}
The set map $t_*$ is injective.
\end{lem}
\begin{proof}
If  $(\{\rho_R \}_{R \in c(a)}, a) \in \Hbar(S,\partial S)$, then
$a$ is the unique multicurve such that  $t_*((\{\rho_R \}_{R \in c(a)},a))( \gamma)=\infty$
if and only if $\gamma$ intersects $a$ essentially (for all 
$\gamma\in\mathcal{C}$). Therefore, the image of $t_*$ determines the multicurve.
Theorem 1.3 in \cite{HP} implies that for each component $R\in c(a)$,
$\rho_R\in H(R,\partial R)$ is determined by  the restriction of 
$t_*((\{\rho_R \}_{R \in c(a)},a))$ to the set of conjugacy classes of elements of
$\pi_1(R)$. Therefore, $(\{\rho_R \}_{R \in c(a)}, a) $ is entirely determined by
$t_*((\{\rho_R \}_{R \in c(a)},a))$.
\end{proof}

Topologize $\Hbar(S,\partial S)$ using $t_*$ as a subspace of $({\bf CP}^1)^\mathcal{C}$.  Let us temporarily call the resulting topological space $t_*(\Hbar(S,\partial S))$.  As a subspace of a metric space, $t_*(\Hbar(S,\partial S))$ is Hausdorff and second countable.  To understand this topology better, we now give a more intrinsic formulation of its convergence.

\begin{defn}\label{defn:shattering}
A sequence $\{ \rho_n \} \subset \operatorname{AH}(S, \partial S)$ is a shattering sequence with shattering multicurve $a$ if:
\begin{enumerate}
\item\label{s1} For each component $F$ of $S - \mathcal{N} a$ the restrictions $\{\rho_{n}|_F\}$ converge in $\AH(F)$.
\item\label{s2} If $\Gamma$ is a nontrivial subgroup of $\pi_1 (S)$ such that the restrictions $\{\rho_{n}|_\Gamma\}$ converge (up to conjugacy) on a subsequence of
$\{\rho_{n}\}$ then $\Gamma$ is conjugate to a subgroup of $\pi_1 (F)$ for some component $F$ of $S - \mathcal{N} a$.
\end{enumerate}
In other words, in Theorem \ref{owbt} there is no need to pass to a subsequence, and $R$ is the complement of a multicurve.
\end{defn}

A sequence $(\{\rho^n_{R}\}_{R \in c(a_n)},a_n)$ in
$\Hbar(S,\partial S)$ is a {\em stable sequence} if the multicurve 
$a_n$ is constant, i.e. there exists a multicurve $a_{stable}$, called the 
{\em stable multicurve}, such that
$a_n=a_{\text{stable}}$ for all $n$.

\begin{defn}\label{defn:stablealgcon}
A stable sequence
$\{(\{\rho^n_{R}\}_{R\in c(a_{\text{stable}})}, a_{\text{stable}})\}$ in
$ \overline{\text{H}}(S,\partial S)$
converges algebraically to $(\{\rho_F\}_{F\in c(a)},a)$ if
\begin{enumerate}
\item $a_{\text{stable}}\subseteq a$. 
\item If $F\in c(a)$, then $\{\rho^n_R|_{\pi_1(F)}\}$ converges to $\rho_F$
in $\AH(F)$
(where $R$ is the element of $c(a_{\text{stable}})$ containing $F$).
\item For all $R\in c(a_{\text{stable}})$, $\{\rho^n_{R}\}$ is a shattering sequence  
in $\AH(R,\partial R)$ with shattering multicurve $a\cap R$. 
\end{enumerate}
\end{defn}

In particular, a sequence $\{(\rho^n_S,\emptyset) \}$ in $\Hbar(S,\partial S)$ converges algebraically to $(\{\rho_F\}_{F\in c(a)},a)$ if and only if it is a shattering sequence with shattering multicurve $a$ and $\{\rho^n_S|_{\pi_1(F)}\}$ converges
algebraically to $\rho_F$ for all $F\in c(a)$.

A (not necessarily stable) sequence in $\Hbar(S,\partial S)$ converges algebraically to $(\{\rho_F\}_{F\in c(a)},a) \in \Hbar(S,\partial S)$ if, after possibly discarding finitely many elements, it can be partitioned into stable subsequences, with distinct stable multicurves, all converging algebraically to $(\{\rho_F \}_{F\in c(a)},a)$.  (Since the stable multicurve for any stable subsequence lies in $a$,
this partition must be finite.) 

\begin{prop}\label{prop:equaltop}
A sequence in $\Hbar(S,\partial S)$ converges in $t_* (\Hbar(S,\partial S))$ if and only if it converges algebraically.
\end{prop}
\begin{proof}
Suppose a sequence in $\Hbar(S,\partial S)$ converges algebraically to $( \{\rho_R\}_{R \in c(a)}, a)$.  Without loss of generality, we can assume the sequence is stable.  Then condition (\ref{s2}) of Definition \ref{defn:shattering} guarantees that any curve essentially intersecting $a$ has trace going to $\infty$.  Condition (2) of Definition \ref{defn:stablealgcon}
guarantees the other traces converge, establishing convergence in $t_* (\Hbar(S,\partial S))$.

Next suppose a sequence $\left( \{\rho^n_R \}_{R \in c(a_n)}, a_n \right)$ converges to $( \{\rho_R \}_{R \in c(a)}, a)$ in $t_* (\Hbar(S,\partial S))$.  We must first check that, after possibly discarding finitely many terms, the sequence can be partitioned into stable sequences.  To begin, suppose there is a subsequence (denoted without subscripts) such that $a_n$ always essentially intersects $a$.  
Then, up to subsequence, $i(a_n, a^0) > 0$
for some component $a^0$ of $a$, so
$t_*( \{\rho_R \}_{R \in c(a)}, a)(a^0)=\infty$, since
$t_*( \{\rho^n_R \}_{R \in c(a_n)}, a_n)(a^0)=\infty$ for all $n$, which is 
a contradiction.  Therefore $i(a_n, a) = 0$ for all $n \gg 0$.

We next claim that $a_n\subset a$ for all $n \gg 0$. If not,
we may pass to a subsequence and choose a component $a_n^0$ of $a_n-a$
such that $\{a_n^0\}$,
viewed as a sequence of projective measured
laminations, converges to a (projective class of a)
measured lamination $\lambda$ on $S$. 
See \cite[Sec. 4]{BoTame} for a discussion of measured laminations and intersection
number. Since each $a_n^0$ is a simple closed curve in $S-a$, $\lambda$ is
disjoint from $a$. Pick an essential simple closed curve $b$ in $S - \mathcal{N} a$ such that $i(b, \lambda) > 0$.  Then by continuity of intersection number, for all $n \gg 0$ we have $i(b, a_n) >0$, implying that  $t_*( \{\rho_R \}_{R \in c(a)}, a)(b)=\infty$.
This is a contradiction, proving that $a_n \subseteq a$ for $n \gg 0$.

After discarding finitely many terms the sequence can therefore be partitioned into a finite set of stable sequences. Suppose that $F\in c(a)$. On any stable subsequence,
with associated multicurve $a_{\text{stable}}$, there exists $R\in c(a_{\text{stable}})$
such that   $F\subset R$. Since $\tr^2(\rho^n_R(c))$ converges to $\tr^2(\rho_F(c))$
for any conjugacy class $c$ of an element of $\pi_1(F)$, it follows that
$\{\rho^n_R|_{\pi_1(F)}\}$ converges to $\rho_F$ in $AH(F)$ (see
Corollary 2.3 in \cite{HP}).
From here the definition of our extended trace function implies condition (\ref{s2}) of Definition \ref{defn:shattering} for each stable subsequence.
\end{proof}

With their equivalence established, we refer to the topology on
$t_* (\Hbar(S,\partial S))$ as the algebraic topology and denote it by $\AHbar (S,\partial S)$.
This topology is closely related to the notion of algebraic convergence on subsurfaces which played a role in the proof of the Ending Lamination Conjecture \cite[Sec, 6]{BCM}.
One expects the augmented space to be at least as topologically complicated as $\AH(S,\partial S)$, which is known, for example, not to be locally connected \cite{BrombergPT,Magid}.

\section{The augmented moduli space} \label{sec:ms}

The goal of this section is to define the quotient augmented moduli space and establish that it is sequentially  compact.

If $\phi\in \Mod(S)$
and $(\{\rho_F\}_{F\in c(a)},a)\in \AHbar (S,\partial S)$, then we can choose a representative, also called $\phi$, so that $\phi(a)$ is a (geodesic) multicurve.
The mapping class $\phi$ takes $(\{\rho_F\}_{F\in c(a)},a)$ to
$(\{\rho_F\circ \phi_*^{-1}\}_{\phi(F)\in c(\phi(a))},\phi(a)\})$.
Stated geometrically, $\phi$ takes $(N,a,m)$ to $(N,\phi(a),m\circ \phi^{-1}|_{S-\mathcal{N}\phi(a)})$.
It is easy to check that the each mapping class induces a homeomorphism of 
$\AHbar (S,\partial S)$
and that we obtain a continuous extension of the action of $\text{Mod}(S)$ on 
$\AH(S,\partial S)$.

We then define the natural quotient space, with its induced quotient algebraic
topology by
\[ \AIbar (S,\partial S) 
    := {\AHbar (S,\partial S) }/ \Mod(S).\]
The key feature of our augmented moduli space is that it is sequentially compact.

\begin{thm} \label{compactness}
$\AIbar (S,\partial S)$ is a sequentially compact topological space.
\end{thm}

This compactness result is essentially a corollary of a result of Canary-Minsky-Taylor, 
together with Theorem \ref{owbt}.  We restate the result of Canary-Minsky-Taylor in the setting of $\AH (S,\partial S)$.

\begin{thm} \label{CMT compactness} \cite[Thm.5.5]{CMT}
Let $\{\rho_n\}$ be a sequence in $\AH (S,\partial S)$.  Then there exists a 
subsequence $\{\rho_j\}$, a sequence $\{\phi_j\}$ in $\Mod(S)$, and a multicurve
$a$ in $S$, such that if $R$ is a component of $S-\mathcal{N}a$ then 
$\{\rho_j\circ(\phi_j)_*^{-1} |_{\pi_1(R)}\}$ converges in $\AH(R)$ to an element of
$\AH(R,\partial R)$.
\end{thm}

If we combine the above result with Theorem \ref{owbt}, we
see that we may always re-mark a subsequence to find a shattering subsequence.
The proof consists simply of applying Theorem \ref{owbt} to the sequence
produced by Theorem \ref{CMT compactness}.  The multicurve $b$ in the
statement of Corollary \ref{wwn} is always contained in the multicurve
associated to the subsequence produced by Theorem \ref{CMT compactness}.

\begin{cor}\label{wwn}
Let $\{\rho_n\}$ be a sequence in $\AH(S,\partial S)$.  Then there exists a 
subsequence $\{\rho_j\}$, a sequence $\{\phi_j\}$ in $\Mod(S)$, and a multicurve
$b$ in $S$, such that $\{\rho_j\circ (\phi_j)_*^{-1}\}$ is a shattering sequence with
shattering multicurve $b$.
\end{cor}

\begin{proof}[Proof of Theorem \ref{compactness}]
Let $\{(\{\rho^n_{R_n}\}_{R_n\in c(a_n)},a_n)\}$ be a sequence in
$\AHbar (S,\partial S)$.  Since there exists only a
finite number of multicurves, 
up to homeomorphism,  on a compact surface, we may pass to a subsequence
$(\{\rho^j_{R_j}\}_{R_j\in c(a_j)},a_j)$ and find a sequence $\{\phi_j\}$ in
$\Mod(S)$ so that $\phi_j(a_j)$ is the same multicurve, say $a$, for all $j$.
Then  each $\phi_j(\{\rho^j_{R_j}\}_{R_j\in c(a_j)},a_j)$ can be rewritten as
$(\{\sigma^j_F\}_{F\in c(a)},a)$.  We then apply Corollary
\ref{wwn} to the sequence $\{\sigma^j_F\}$ successively for each component $F$
of $c(a)$,  possibly further remarking the surface $F$ in the process, 
to find an algebraically convergent subsequence of $\{(\{\sigma^j_F\}_{F\in c(a)},a)\}$.
Since every sequence in $\AHbar (S,\partial S)$ has a subsequence
which can be re-marked by elements of $\Mod (S)$ so that it converges in
$\AHbar (S,\partial S)$, we conclude that
$\AIbar (S,\partial S)$ is sequentially compact.
\end{proof}

\medskip\noindent
{\bf Remark.} Geometrically finite points in
$\AIbar (S,\partial S)$ are not necessarily closed. (Recall 
Proposition \ref{prop:T1} proved that geometrically finite points are closed
in $\AI (S,\partial S)$.) To see this, consider the following example.
Let $b \subset S$ be a simple closed separating geodesic.  
Let $\rho \in \H(S,\partial S)$ be geometrically finite with exactly one rank one parabolic corresponding to the curve $b \subset S$. 
Let $D: S \to S$ indicate a Dehn twist along $b$. 
Consider the two component element $( \{\rho|_R \}_{R \in c(b)}, b) \in \Hbar(S,\partial S)$.
Clearly $\rho$ and $( \{\rho|_R \}_{R \in c(b)}, b)$ project to distinct points in
$\AIbar(S,\partial S)$.  Nonetheless, the sequence $\{ \rho \circ D^n \}$ converges algebraically to $( \{\rho|_R \}_{R \in c(b)}, b)$.  This shows the projection of $\rho$ to $\AIbar(S,\partial S)$ is not a closed point.

\section{The Fuchsian locus and the Deligne-Mumford compactification}

One forms the augmented Teichm\"uller space $\overline{\mathcal{T}}(S)$ by
appending  the Teichm\"uller space $\mathcal{T}(S-\mathcal{N}a)$ associated to
the complement of every multicurve $a$ on $S$. Notice that one may
associate to any point in $\mathcal{T}(S-\mathcal{N}a)$ a unique Fuchsian
element $(\{\rho_R\}_{R \in c(a)}, a)$ of $\Hbar(S,\partial S)$. Therefore one
may identify $\overline{\mathcal{T}}(S)$ with the set
$\overline{\mathcal{F}}(S, \partial S) \subset \Hbar(S, \partial S)$ of 
Fuchsian elements of $\Hbar(S,\partial S)$.   We call
$\overline{\mathcal{F}}(S, \partial S)$ the Fuchsian locus. One may check that
$\overline{\mathcal{F}}(S, \partial S)$ (with the algebraic topology) is actually
homeomorphic to $\overline{\mathcal{T}}(S)$. 
Since the Deligne-Mumford compactification  $\overline{\mathcal{M}}(S)$ of
$\mathcal{M}(S)$ arises as the quotient
of the augmented Teichm\"uller space $\overline{\mathcal{T}}(S)$ under the
action of $\Mod(S)$, we may
identify $\overline{\mathcal{M}}(S)$ with the quotient of
the Fuchsian locus $\overline{\mathcal{F}}(S, \partial S)$ in
$\Ibar (S,\partial S)$.
We refer the reader to Wolpert's survey article \cite{Wolpert} for a discussion of the
basic properties of augmented Teichm{\"u}ller space
and its relationship to the Deligne-Mumford compactification.

It is easy to check that
the Fuchsian locus is closed in $\AHbar (S,\partial S)$ and
invariant under the action of $\Mod(S)$, implying that 
$\overline{\mathcal{M}}(S)$ is identified with a closed subset of
$\AIbar (S,\partial S)$.

\begin{prop} \label{FE = DM}
The natural  embedding $\iota: \mathcal{M}(S) \to \AI(S, \partial S)$, sending a hyperbolic surface to its corresponding Fuchsian $3$-manifold, extends to an
embedding of the Deligne-Mumford compactification $\overline{\mathcal{M}}(S) \to \AIbar(S, \partial S)$ with image the set of Fuchsian $3$-manifolds.
\end{prop}

\medskip\noindent
{\bf Remark:} We recall that if the augmented Teichm\"uller space is not a point,
then it fails to be locally compact. Since $\overline{\mathcal{T}}(S)$ is homeomorphic
to a closed subset of $\AHbar(S,\partial S)$, it follows that $\AHbar(S,\partial S)$
also fails to be locally compact.

\section{Other topologies}

In \cite{Th3}, Thurston discusses two other topologies on $\H(S,\partial S)$,
the strong topology and the quasi-isometric topology. Both 
extend naturally to topologies on $\Hbar(S,\partial S)$.

A sequence $\{ (N_n,m_n)\}$ in $\H(S,\partial S)$ converges strongly to 
$(N,m)\in \H(S,\partial S)$ if there exists a sequence 
$\{ h_n:N\to N_n\}$ of homotopy equivalences which $C^\infty$-converge to an isometry
on every compact subset of $N$ such that $h_n\circ m$ is homotopic to $m_n$ for all
$n$.  The key difference with the definition of algebraic convergence
is that $\{ h_n\}$ converges to an isometry, rather than just a local isometry.
The deformation space $\H(S,\partial S)$ equipped with the
strong topology is denoted $\GH(S,\partial S)$. 
The ``G" is due to Thurston, who called this the geometric topology. 

We may readily generalize this to the setting of the augmented deformation
space $\Hbar(S,\partial S)$. We say that a sequence $\{(N_n,a_ n,m_n)\}$
in $\Hbar(S,\partial S)$ converges strongly to $(N,a,m)$ if 
$\{(N_n,a_ n,m_n)\}$ converges algebraically to $(N,a,m)$
and there
exists a sequence of continuous maps $\{h_n:N\to N_n\}$ such that $h_n\circ m$ is homotopic
to $m_n|_F$, for all $n$, on every component $F$ of $S-\mathcal{N}a$, and
$\{h_n\}$ $C^\infty$-converges to an isometry on every compact subset of $N$.
The deformation space $\Hbar(S,\partial S)$ equipped with the
strong topology is denoted $\GHbar(S,\partial S)$. 

A sequence $\{ (N_n,m_n)\}$ in $\H(S,\partial S)$ converges in the quasi-isometric
topology to 
$(N,m)\in \H(S,\partial S)$ if, for all large enough $n$, there exists a 
$K_n$-bilipschitz diffeomorphism $h_n:N\to N_n$
such that $h_n\circ m$ is homotopic to $m_n$  with $\lim K_n=1$.
The deformation space $\H(S,\partial S)$ equipped with the
quasi-isometric topology is denoted $\QH(S,\partial S)$. More generally,
a sequence $\{(N_n,a_ n,m_n)\}$
in $\Hbar(S,\partial S)$ converges in the quasi-isometric topology to $(N,a,m)$
if, for all large enough $n$, $a_n=a$ and there exists a $K_n$-bilipschitz 
diffeomorphism $h_n:N\to N_n$ such that $h_n\circ m$ is homotopic
to $m_n$ and $\lim K_n=1$.
The deformation space $\Hbar(S,\partial S)$ equipped with the
quasi-isometric topology is denoted $\QHbar(S,\partial S)$. 

One may readily check that $\Mod(S)$ acts on
both $\GHbar(S,\partial S)$ 
and $\QHbar(S,\partial S)$
as a group of homeomorphisms. So one obtains strong and quasi-isometric
topologies on the quotient space $\Ibar(S,\partial S)$.

\begin{align*}
\GIbar(S,\partial S) & := 
    {\GHbar(S,\partial S)}/ \Mod(S), \\
\QIbar(S,\partial S) & := 
    {\QHbar(S,\partial S) }/ \Mod(S).
\end{align*}

It follows from the analogous fact for $\H(S,\partial S)$ that the identity maps
\[ \QIbar (S,\partial S) \to
    \GIbar (S,\partial S) \to
    \AIbar (S,\partial S)\]
are continuous, but the inverse maps are not (see \cite{Th3}).

The space $\QIbar (S,\partial S)$ is locally nice and globally terrible.  If $S$ is not a thrice-punctured sphere, it is a disjoint union of an uncountable collection of noncompact orbifolds of various dimensions and an uncountable number of isolated points.  (This follows from Sullivan's extension of the Quasiconformal Parametrization Theorem, also known as Sullivan rigidity \cite{Su}.)
In particular, $\QIbar(S,\partial S)$ is Hausdorff and noncompact.

\begin{prop}
The space $\GIbar(S,\partial S)$ is not sequentially compact.
\end{prop}
\begin{proof}
Consider the sequence $\{\rho_n\}$ from the proof of Proposition \ref{prop:!T2}.
It determines a sequence of hyperbolic manifolds $N_n$ whose geometric
limit $X$ is homeomorphic to $S \times (0,1)$ minus $b \times \{1/2\}$, where
$b$ is a simple closed curve of $S$ \cite{KT}.  
The manifold $X$ is not homotopy equivalent to $S$.  No matter how the
sequence $N_n$ is marked, this geometric limit will not change.  This implies
that no subsequence converges in $\GIbar (S,\partial S)$.
\end{proof}
Finally, the examples in Proposition \ref{T_1-separated} also converge strongly, see
Theorem 3.12 in \cite{Mc3}, so we see that there are points in
$\GIbar(S,\partial S)$ that are not closed.

\bigskip

One may define a refinement of the algebraic topology on 
$\Hbar(S,\partial S)$, which is still coarser than the strong topology,
so that geometrically finite points are closed in the resulting quotient
topology on $\Ibar (S,\partial S)$, yet the resulting
quotient topology on the augmented moduli space is still sequentially compact. 
We say that a stable sequence 
$\{(\{\rho^n_R\}_{R\in c(a_{\text{stable}})},a_{\text{stable}})\}$
converges maximally algebraically to $(\{\rho_F\}_{F\in c(a)},a)$ if
it converges algebraically and there does not exist a subsequence 
$\{(\{\rho^j_R\}_{R\in c(a_{\text{stable}})},a_{\text{stable}})\}$
and a sequence  $\{ \phi_j\}$ in $\Mod(S)$, each of which is
a product of Dehn twists about elements of $a-a_{\text{stable}}$,
such that $\{\phi_j(\{\rho^j_R\}_{R\in c(a_{\text{stable}})},a_{\text{stable}})\}$
converges algebraically to $(\{\rho_B\}_{F\in c(b)},b)$ where $b$ is a proper
subset of $a$. We denote $\Hbar(S,\partial S)$ with the topology
of maximally algebraic convergence by $\operatorname{B\overline{H}}(S,\partial S)$.
Its quotient by the action of $\Mod(S)$ is denoted by
$\operatorname{B\overline{\mathcal{I} } }(S,\partial S)$. (Note that this topology
is designed specifically to disallow examples like those described in
the remark terminating Section \ref{sec:ms}.)  The proof of Theorem \ref{compactness}
can be easily modified 
to verify the sequential compactness of $\operatorname{B\overline{\mathcal{I} } }(S,\partial S)$.

\bibliography{bibliography_revise}
\end{document}